\documentclass[MSNbibl,number,citesort,dvips]{arxbj}
\usepackage{mathbh,upgreek}

\aid{0}
\volume{20}
\issue{4}
\pubyear{2014}
\firstpage{1802}
\lastpage{1818}
\doi{10.3150/13-BEJ542} 

\makeatletter
\newcommand{\rrVert}{\Vert}
\newcommand{\rrvert}{\vert}
\newcommand{\llVert}{\Vert}
\newcommand{\llvert}{\vert}
\newtheorem{theorem}{Theorem}[section]

\newtheorem{lemma}[theorem]{Lemma}

\newremark{remark}{Remark}[section]
\newremark{example}[theorem]{Example}

\newcommand{\one}{{\mathbh{1}}}

\newcommand{\ee} {\mathbb{E}}
\newcommand{\FF} {\mathcal{F}}

\newcommand{\KK} {\mathcal{K}}
\newcommand{\LL} {\mathcal{L}}

\newcommand{\PP} {\mathcal{P}}
\newcommand{\rr} {\mathbb{R}}

\newcommand{\fourier}{{\mathcal{T}}}

\makeatother

\begin{document}
\begin{frontmatter}

\title{Minimax bounds for estimation of normal mixtures}
\runtitle{Minimax bounds for estimation of normal mixtures}

\begin{aug}
\author{\inits{A.K.H.}\fnms{Arlene K.H.} \snm{Kim}\ead[label=e1]{a.kim@statslab.cam.ac.uk}}
\address{Statistical Laboratory, Center for Mathematical Sciences,
 University of Cambridge, Wilberforce Road, Cambridge, CB30WB, UK. \printead{e1}}
\end{aug}

\received{\smonth{10} \syear{2012}}
\revised{\smonth{4} \syear{2013}}

%
\begin{abstract}
This paper deals with minimax rates of convergence for estimation of
density functions on the real line. The densities are assumed to be
location mixtures of normals, a global regularity requirement that
creates subtle difficulties for the application of standard minimax
lower bound methods.
Using novel Fourier and Hermite polynomial techniques, we determine the
minimax optimal rate -- slightly larger than the parametric rate -- under
squared error loss.
For Hellinger loss, we provide a minimax lower bound using ideas
modified from the squared error loss case.
\end{abstract}

%

\end{frontmatter}

\section{Introduction} \label{intro}
This paper establishes the optimal minimax rate of convergence under
squared error loss, for densities that are normal mixtures. The
analysis reveals a subtle difficulty in the application of Assouad's
lemma to parameter spaces defined by indirect regularity conditions,
which complicate the usual construction of subsets of the parameter
space indexed by ``hyper-rectangles.''

More precisely, we consider independent observations from probability
distributions $P_f$ on the real line whose densities $f$ (with respect
to Lebesgue measure on $\rr$) belong to the set of convolutions
\[
\FF= \biggl\{f\dvt f(x)= \phi\star\Pi(x) = \int\phi(x-u) \,\mathrm{d}\Pi
(u), \Pi\in\PP(
\rr) \biggr\},
\]
where $\phi$ denotes the standard normal $N(0,1)$ density and $\PP(\rr)$
denotes the set of all probability measures on the (Borel sigma-field
of the) real line.
Our main result gives an asymptotic minimax lower bound for the $L_2$
risk of estimators of $f \in\FF$.

%
\begin{theorem}\label{mainthm}
Let $X_1, \ldots, X_n$ be independent and identically distributed with
density $f \in\FF$.
Then there exists a positive constant $c$ such that
\[
\sup_{f \in\FF} \ee_{n,f} \int_{-\infty}^\infty
\bigl(\hat f_n(x)-f(x) \bigr)^2\,\mathrm{d}x \ge c \cdot\log n \cdot
\frac{1}{n\sqrt{\log n}}:=c\ell_n 
\]
for every estimator $\hat f_n = \hat f_n(X_1,\ldots, X_n)$.
\end{theorem}

Let $\FF_0$ denote the subset of $\FF$ consisting of those normal
mixture densities whose mixing measure is absolutely continuous with
respect to Lebesgue measure.
The proof of Theorem \ref{mainthm}, which is given in Section~\ref{secoutline}, involves the construction of a finite subset of $\FF_0$,
so the lower bound also holds when the supremum is taken over $f \in
\FF_0$.
Perhaps the most interesting feature of this result is that the same
rate has been obtained as an upper bound for the minimax risk with
respect to squared error loss over much larger classes of functions.
For instance, \cite{ibragimov2001} defined the class $\FF^*$
consisting of those densities that can be extended to an entire
function $f^*$ on $\mathbb{C}$ satisfying $\sup_{y \in\rr} \mathrm{e}^{-y^2/2}
\sup_{x \in\rr}|f^*(x+\mathrm{i}y)| < \infty$.
He proved the following theorem.
%
%
\begin{theorem}[(\cite{ibragimov2001}, Theorem 4.1)]\label{ibragimovthm}
Let $X_1, \ldots, X_n$ be independent and identically distributed with
density $f \in\FF^*$.
Then there exists an estimator $\hat f_n = \hat f_n(X_1, \ldots, X_n)$
of $f$ such that
\[
\sup_{ f \in\FF^*} \ee_{n,f} \int_{-\infty}^\infty
\bigl( \hat f_n(x)-f(x) \bigr)^2 \,\mathrm{d}x = \mathrm{O}(
\ell_n).
\]
\end{theorem}
For the reader's convenience, in Section~\ref{upperbound}, we show that $\FF\subseteq
\FF^*$ and summarize Ibragimov's proof.
Theorems \ref{mainthm} and \ref{ibragimovthm} together establish that
the minimax optimal rate of estimation for squared $L_2$ loss is $\ell
_n$ for any class of functions containing $\FF_0$ and contained in
$\FF^*$.
In particular, this is the case for $\FF$.

While a minimax result under the $L_2$ loss presents the most
successful case, this loss function is often criticized for giving too
little weight to errors from the tails.
As an alternative, we also consider the Hellinger loss.
Define a class of probability measures with sub-Gaussian tails,
\[
\PP_s(\rr) := \bigl\{ \Pi\in\PP(\rr)\dvt\exists C>0 \mbox{ such
that } \Pi\bigl(|u|>t\bigr) \leq C \exp\bigl(-t^2/C\bigr) \mbox{ for all real } t
\bigr\}.
\]
For the following class of normal location mixtures
\[
\FF_s := \biggl\{ f\dvt f(x) = \phi\star\Pi(x)= \int\phi(x-u) \,\mathrm{d}
\Pi(u), \Pi\in\PP_s(\rr) \biggr\},
\]
\cite{ghosalvaart2001} provide a sieved maximum likelihood estimator
whose convergence rate is $\mathrm{O}((\log n)^2/n)$.
However, as they pointed out, the optimal rate for $\FF_s$ is still unknown.
Our technique gives a lower bound that lies within a logarithmic factor
of Ghosal and van der Vaart's upper bound.

%
\begin{theorem}\label{mainthm2}
Let $X_1,\ldots,X_n$ be independent and identically distributed with
density $f \in\FF_s$. Then there exists a positive constant $c$ such that
\[
\sup_{f \in\FF_s} \ee_{n,f} \int_{-\infty}^\infty
\bigl( \sqrt{\hat f_n(x)}- \sqrt{f(x)} \bigr)^2\,\mathrm{d}x \ge c
\cdot\log n \cdot\frac{1}{n} 
\]
for every estimator $\hat f_n = \hat f_n(X_1, \ldots, X_n)$.
\end{theorem}
%

To prove Theorems \ref{mainthm} and \ref{mainthm2}, we use a variation on Assouad's lemma
(cf. \cite{vanassouad}, page 347).
When specialized to density estimation, the lemma can be cast into the
following form. (Henceforth, we omit the $\pm\infty$ terminals on the
integrals when there is no ambiguity.) For completeness, we provide the
proof in the \hyperref[app]{Appendix}.

%
\begin{lemma}\label{assouadsim}
Let $\{ f_\alpha, \alpha\in\{0,1\}^K\} \subseteq\FF$ where $K$ is
a finite index set of cardinality $m$. Suppose~$W$ is a nonnegative
loss function for which there exists $\zeta>0$ such that, for all $g_1,
g_2 \in\FF$,
%
%
\begin{equation}\label{metricW}
\inf_{f \in\FF} W(f,g_1) + W(f,g_2) \geq\zeta W(g_1,g_2).
\end{equation}
Suppose also that for some constants $c_0>0$ and $1>c_1>0$,
%
%
\begin{equation}
W(f_\alpha, f_\beta) \geq c_0 \varepsilon^2 \Vert\alpha-\beta\Vert_0
\qquad\mbox{for all } \alpha, \beta\in\{0,1\}^K \label{cond1}
\end{equation}
and
%
%
\begin{equation}\int\frac{(f_\alpha-f_\beta)^2}{ f_\alpha} \leq\frac
{c_1}{n} \qquad\mbox{if } \Vert\alpha-\beta\Vert_0=1, \label{cond2}
\end{equation}
where $\Vert\alpha-\beta\Vert_0= \sum_{k \in K} \one\{ \alpha_k \neq
\beta_k\}
$, the Hamming distance.
Then, for every estimator $\hat f_n$ based on $n$ independent observations,
%
%
\begin{equation}\label{minimaxlower}
\sup_{f \in\FF} \ee_{n,f} W(\hat f_n, f) \geq\frac{c_0 \zeta
}{4}(1-\sqrt{c_1}) m \varepsilon^2.
\end{equation}
%
\end{lemma}
%
%
\begin{remark}
Assumption (\ref{cond2}) regarding the $\chi^2$ distance is merely a
convenient way to show that the testing affinity, $\Vert P_{f_\alpha}^n
\wedge P_{f_\beta}^n\Vert_1$, is at least $1-\sqrt{c_1}$, where $P_{f}^n$
is a product probability measure under $f$ and $\Vert P \wedge Q\Vert
_1$ is
defined as $\int \min(\mathrm{d}P,\mathrm{d}Q)$.
\end{remark}
%
%
\begin{remark}
To apply Lemma \ref{assouadsim}, we try to maximize $m\varepsilon^2$ for
the best possible lower bound.
While we construct the finite density class satisfying the loss
separation condition (\ref{cond1}), we need to restrict the size
$\varepsilon^2$ and $m$ so that two nearest densities should be reasonably
close as in (\ref{cond2}), and so that the constructed densities are
truly in the parameter space $\FF$.
\end{remark}
For the proof in Section~\ref{secoutline}, we construct $f_\alpha$'s of
the form
\[
f_\alpha(x) =f_0(x) + \varepsilon\sum
_{k\in K} \alpha_k \Delta_k(x),\qquad\alpha\in
\{0,1\}^K,
\]
where $f_0$ is the normal density function with a zero mean (and
variance specified later), where $K = \{1,3,\dots,2m-1\}$, and where
$m$, $\varepsilon>0$, and $\Delta_k$ could depend on $n$.
The main difficulty lies in choosing the (signed) perturbations $\Delta
_k$ so that each $f_\alpha$ is a normal location mixture. The natural
way around this problem is to construct the Assouad hyper-rectangle in
the space of mixing distributions,
\[
f_\alpha=\phi\star\Pi_\alpha,\qquad\mbox{where } \Pi_\alpha(u)
= \Pi_0(u) + \varepsilon\sum_{k\in K}
\alpha_k V_k(u), \alpha\in\{0,1\}^K,
\]
where the signed measures $V_k$ must be chosen so that each $\Pi
_\alpha
$ is a probability measure. In contrast to the standard construction,
the indirect form of $f_\alpha= \phi\star\Pi_\alpha$ leads to an
embedding condition of the form
%
%
\begin{equation}\label{nearortho}
W(\phi\star\Pi_\alpha, \phi\star\Pi_\beta) \geq\tau_n\sum_{k
\in K}
(\alpha_k-\beta_k)^2
\end{equation}
for some $\tau_n$.
The right side of (\ref{nearortho}) is expressed in terms of $\sum
_{k\in K}(\alpha_k-\beta_k)^2$ instead of the Hamming distance, in
order to emphasize the orthogonality relation.
If the convolution with the normal density were not present, such a
property could be obtained by choosing the perturbations to be exactly
orthogonal to each other, subject to various other regularity
properties that define the parameter space.
The smoothing effect of the convolution operation, however, makes it
difficult to choose the $V_k$ to achieve such near-orthogonality.
Nevertheless, we can achieve (\ref{nearortho}) by choosing the
perturbations so that their Fourier transforms are orthogonal as
elements in $L_2(\phi^2)$, the space of complex-valued functions $g$
such that $\int\phi(x)^2 |g(x)|^2 \,\mathrm{d}x < \infty$ for the $L_2$ loss.
Similarly, we achieve (\ref{nearortho}) under the Hellinger loss using
the similar ideas under $L_2$ except that $\phi^2$ is replaced by a
different weight function.

\section{Proofs of the lower bounds}\label{secoutline}
First, we introduce some notation used in this section.
We let $\phi_{\sigma^2}$ be the normal density with mean zero and
variance $\sigma^2$.
Following, for example, \cite{rudinrealcomplex}, Chapter~9, we define
the Fourier transform $\fourier$ by
\[
\fourier f(t) := \breve f(t) =\frac{1}{\sqrt{2\uppi}} \int_{-\infty
}^\infty
\exp(- \mathrm{i} xt) f(x)\,\mathrm{d}x
\]
for $f \in\LL_1(\lambda)$ where $\lambda$ is Lebesgue measure, and
then extend from $\LL_1 \cap\LL_2$ to $\LL_2$ by extending an isometry
of $\LL_1 \cap\LL_2$ into $\LL_2$ to an isometry of $\LL_2$ onto
$\LL_2$.

For both theorems, we construct the signed measures $V_k$ to have
(signed) densities $v_k$ with respect to $\lambda$:
%
%
\begin{equation}\label{formG}
\pi_\alpha(u) = \frac{\mathrm{d}\Pi_\alpha}{\mathrm{d}\lambda}(u)
= \pi_0(u)+ \varepsilon\sum_{k\in K}\alpha_k v_k(u), \qquad\alpha\in\{
0,1\}^K,
\end{equation}
where $\pi_0$ is the normal density with zero mean and each $v_k$ is a
function for which $\int v_k=0$ and
\[
\pi_0(u) +\varepsilon\sum_{k\in K}
\alpha_k v_k(u) \ge0\qquad \mbox{for all $u$}.
\]
We then need to check the assumptions for Lemma \ref{assouadsim}.

\subsection{Ideas in the proof of Theorem \texorpdfstring{\protect\ref{mainthm}}{1.1}}
Here we let $W(f,g) := \Vert f-g\Vert_2^2 = \int(f-g)^2$, so (\ref{metricW})
is satisfied with $\zeta=1/2$.
The choice of the $v_k$'s is suggested by Fourier methods.
By the Plancherel formula (and the fact that $\breve\phi=\phi$),
recalling that $f_\alpha= \phi\star\Pi_\alpha$,
\[
\frac{1}{2\uppi}\Vert f_\alpha-f_\beta\Vert
_2^2 = \frac{1}{2\uppi}\Vert\breve f_\alpha-
\breve f_\beta\Vert_2^2 = \varepsilon^2
\int_{-\infty}^\infty\biggl\llvert\sum
_{k \in K}(\alpha_k-\beta_k) \phi(t)
\breve v_k (t) \biggr\rrvert^2 \,\mathrm{d}t,
\]
which lets us write the desired property (\ref{cond1}) of Lemma \ref
{assouadsim} as
\[
\int_{-\infty}^\infty\biggl\llvert\sum
_{k \in K}(\alpha_k-\beta_k)
\phi(t) \breve v_k (t) \biggr\rrvert^2 \,\mathrm{d}t \geq
\frac{c_0}{2\uppi} \sum_{k \in K} (\alpha_k-
\beta_k)^2 \qquad\mbox{$\forall\alpha, \beta\in\{0,1
\}^K$}.
\]
%
We might achieve such an inequality by choosing the $v_k$'s to make the
functions $\psi_k(t) := \phi(t)\breve v_k (t)$ orthogonal. Ignoring
other requirements for the moment, we could even start from an
orthonormal set $\{\psi_k\}$ and then try to define $v_k$ as the (inverse)
Fourier transform of $\psi_k(t)/\phi(t)$, provided that the ratio is
square integrable.
This heuristic succeeds if we start from the normalized orthogonal
functions (see \cite{fourierjackson}, Chapter~9),
%
%
\begin{equation}\label{psikchoice}
\psi_k(t) =C \mathrm{i}^{-k} \phi(t)^2 \frac{H_{k}(2t)}{\sqrt{k!}}
= \mathrm{i}^{-k} \sqrt{2\phi(2t)} \frac{H_{k}(2t)}{\sqrt{k!}}
\end{equation}
for $k\in K:= \{1,3,\ldots,2m-1\}$,
where
$C= \sqrt{2}(2\uppi)^{3/4}$ is chosen so that $C \phi(t)^2 = \sqrt
{2\phi(2t)}$
and $H_k(t)$ is the Hermite polynomial of order $k$,
the polynomial for which $\phi(t)$ has $k${th} derivative $(-1)^k
H_k(t) \phi(t)$.

%
\begin{remark}
$\{H_k, k=1,2,\ldots\}$ is sometimes called the ``probabilists' Hermite
Polynomials'' (denoted as ``$\mathit{He}$'' in \cite{table}), as opposed to the
``physicists' Hermite Polynomials'' $\mathbf{H}$. There is one-to-one
relation between $H$ and $\mathbf{H}$, given by
%
\[
H_k(t)=2^{-k/2} \mathbf{H}_k \biggl(
\frac{t}{\sqrt{2}} \biggr).
\]
\end{remark}
To calculate the Fourier inverse transform of $\psi_k(t)/\phi(t)$, we
provide the following lemma.
%
%
\begin{lemma}\label{hermiteFT} For $b>a>0$,
%
%
\begin{equation}\fourier^{-1} \bigl[ \phi(a t) H_k(bt) \bigr](u)= Q_k \phi
\biggl(\frac{u}{a} \biggr)H_k\bigl(b'u\bigr),
\end{equation}
where $Q_k = (\mathrm{i}c_{a,b} )^{k}/a$ with $c_{a,b} = \sqrt{b^2/a^2-1}$ and
$b'= b/(a^2 c_{a,b})$.
\end{lemma}
%
%
\begin{remark}
Lemma \ref{hermiteFT} illustrates a general form of the
eigenvalue-eigenfunction relation for the Fourier transform of Hermite
functions,
\[
\fourier\bigl[\phi(t) H_k(\sqrt{2}t)\bigr] (u) = (-\mathrm{i})^k
\phi(u) H_k(\sqrt{2}u).
\]
(See (7.376) in \cite{table}, or for more details, see Section~4.11
in \cite{kawata}).
\end{remark}
We now formulate these arguments into a proof.
\begin{pf*}{Proof of Theorem \ref{mainthm}}
By Lemma \ref{hermiteFT}, defining $\{\psi_k, k \in K\}$ as in (\ref
{psikchoice}) leads to
%
%
\begin{equation}\label{Gk}
v_k(u) = C \sqrt{\frac{3^{k}}{k!}} \phi(u) H_{k} \biggl(\frac
{2}{\sqrt{3}} u \biggr) \qquad\mbox{for $k \in K$},
\end{equation}
because $\fourier^{-1}[\phi(t)H_k(2t)](u) = \mathrm{i}^k 3^{k/2} \phi(u)
H_{k}(2u/\sqrt{3})$.
By restricting to odd values of $k$, we make the $v_k$'s real-valued
and odd, thereby ensuring that $\int v_k \,\mathrm{d}\lambda=0$ and
$\int\pi
_\alpha\,\mathrm{d}\lambda=1$ for each $\alpha$ in $\{0,1\}^K$.

In summary, the choice of $v_k$ as in (\ref{Gk}) gives
%
%
%
\begin{eqnarray}
\frac{1}{2\uppi}\Vert f_\alpha-f_\beta\Vert
_2^2 =\varepsilon^2 \int
_{-\infty}^\infty\biggl( \sum
_{k \in K} (\alpha_k-\beta_k)
\psi_k(t) \biggr)^2 \,\mathrm{d}t =\varepsilon^2 \sum
_{k \in K} (\alpha_k-\beta_k)^2.
\label{l2loss}
\end{eqnarray}
That is, the condition (\ref{cond1}) of Lemma \ref{assouadsim} is
satisfied with $c_0=2\uppi$.

We still need to check the condition (\ref{cond2}), and also show
that $\varepsilon$ can be chosen small enough to make
all the $\pi_\alpha$'s nonnegative.
Actually, we first show that $\pi_\alpha\geq\pi_0/2 >0$ by choosing
%
%
\begin{equation}\label{halfg}
\varepsilon\leq\frac{1}{16}3^{-m+1/2}m^{-3/2},
\end{equation}
and by choosing $\pi_0 = \phi_m$.
Secondly, we determine the largest size $m$ while the two densities
$f_\alpha$ and $f_\beta$ are close in terms of the $\chi^2$ distance as
$\mathrm{O}(1/n)$ when there is only one different coordinate between
$\alpha$
and $\beta$.

To control the denominator in (\ref{cond2}), we first show that
$|v_k(u)| \leq C_k\sqrt{m} \pi_0(u)$ where $C_k = 8 \cdot3^{k/2}$.
By Cram$\acute{\mbox{e}}$r's inequality \cite{table}, equation (8.954)
%
%
\begin{equation}\label{cramer}
\bigl|H_k(u)\bigr| \leq\kappa\sqrt{k!} \exp\bigl(u^2/4\bigr) \qquad\mbox{with } \kappa
\approx1.086.
\end{equation}
Applying this inequality to (\ref{Gk}),
%
%
%
\begin{eqnarray}
\bigl|v_k(u)\bigr| 
&\leq&\kappa C3^{k/2}\frac{1}{\sqrt{2\uppi}}
\exp\biggl(-\frac
{1}{6}u^2 \biggr) \leq C_k
\phi(u/\sqrt{3}) \label{bound1}
\\
&\leq& C_k\phi(u/\sqrt{m}) = C_k \sqrt{m}
\pi_0(u). \label{bound2}
\end{eqnarray}
Using (\ref{bound2}), we have
\begin{eqnarray*}
\pi_\alpha(u) &=& \pi_0(u) + \varepsilon\sum
_{k \in K} \alpha_k v_k(u) \geq
\pi_0(u)- \varepsilon\sum_{k \in K}
C_k \sqrt{m} \pi_0(u)
\\
&\geq&\pi_0(u) \bigl[1- C_{2m-1} m^{3/2}
\varepsilon\bigr]
\\
&=&\pi_0(u) \bigl[1- 8 \cdot3^{m-1/2}m^{3/2}
\varepsilon\bigr] \geq\frac
{\pi_0(u)}{2}
\end{eqnarray*}
by the choice of $\varepsilon$ in (\ref{halfg}).

Hence, under the condition (\ref{halfg}), $\phi\star\Pi_\alpha:=
f_{\alpha} \geq f_0/2:= \phi\star\Pi_0/2$, which implies that the
second condition in Lemma \ref{assouadsim} is rewritten as $\int
(f_\alpha-f_\beta)^2/f_0 \leq c_1/2n$ for $\alpha$ and $\beta$ having
only one different coordinate.
The denominator $f_0 = \phi\star\Pi_0$ is again normally distributed
with mean zero and variance $1+m$ by the choice of $\mathrm{d}\Pi
_0/\mathrm{d}\lambda
:=\pi
_0 = \phi_m$ density.

For convenience, we let $\alpha_{1} \neq\beta_{1}$ (all the other
cases work the same way). By splitting the integral into two regions
$|x| \leq M\sqrt{m}$ and $|x| > M\sqrt{m}$ with a constant $M^2 =
8\log9$,
\begin{eqnarray*}
\int\frac{(f_{\alpha}-f_{\beta})^2}{f_{0}}= \int_{|x| \leq M\sqrt{m}}
\frac{(f_\alpha-f_\beta)^2}{f_0} +
\varepsilon^2 \int_{|x| > M\sqrt{m}} \frac{ (\int\phi(x-u) v_{1}(u)\,
\mathrm{d}\lambda)^2 }{f_0}.
\end{eqnarray*}

For the first integral, the denominator is bounded below on the
interval $\{|x| \leq M\sqrt{m}\}$, since
\[
f_0(x) \bigl\{ |x| \leq M\sqrt{m}\bigr\} > \exp\bigl(-M^2/2
\bigr)/(2\sqrt{2\uppi} \sqrt{m}) := 1/\bigl(C^\ast\sqrt{m}\bigr),
\]
where $C^\ast:= 2\sqrt{2\uppi} \exp(M^2/2)$. Then, using the $\LL_2$
loss calculation from (\ref{l2loss}),
\[
\int_{|x| \leq M\sqrt{m}} \frac{(f_{\alpha}(x)-f_{\beta
}(x))^2}{f_{0}(x)} \,\mathrm{d}x \leq C^\ast
\sqrt{m} \Vert f_{\alpha}-f_{\beta}\Vert_2^2
= 2\uppi C^\ast\sqrt{m} \varepsilon^2.
\]
%

For the second integral, recall that for any $k = 1,3,\ldots,2m-1$, we have
\[
\bigl|v_{k}(u)\bigr| \leq C_{2m-1} \phi(u/\sqrt{3}) := C_{2m-1}
\sigma_0 \phi_{\sigma_0^2}
\]
with $\sigma_0=\sqrt{3}$ as in (\ref{bound1}). Using $C_{2m-1} := 8
\cdot3^{m-1/2}$ and $\phi_{1+\sigma_0^2}(x) \leq\sqrt{m} \phi
_{1+m}(x)$, with a notation $R(x) := \{|x|>M\sqrt{m}\}$, we bound the
second integral:
\begin{eqnarray*}
\varepsilon^2 \int_{R(x)} \frac{ (\int\phi(x-u) v_{1}(u)\,\mathrm
{d}\lambda
)^2}{f_0(x)}\,\mathrm{d}x &
\leq&\varepsilon^2C_{2m-1}^2 \sigma_0^2
\int_{R(x)} \frac{ (
\int
\phi(x-u)\phi_{\sigma_0^2}(u)\,\mathrm{d}\lambda)^2}{\phi_{1+m}(x)}\,
\mathrm{d}x
\\
&\leq&\sqrt{m}\varepsilon^2 C_{2m-1}^2
\sigma_0^2 \int_{R(x)} \phi
_{1+\sigma_0^2}(x) \,\mathrm{d}x
\\
&=& \biggl( \frac{64}{3} \sqrt{m} \varepsilon^2 \biggr) \biggl(
3^{2m} \int_{R(x)} \phi_{1+\sigma_0^2}(x) \,\mathrm{d}x \biggr)
\\
&\leq&\frac{64}{3}\sqrt{m}\varepsilon^2.
\end{eqnarray*}
Here the last inequality is obtained by a Gaussian tail property with
$\sqrt{m} \gg\sigma_0 := \sqrt{3}$, namely
\[
\int_{|x|>M\sqrt{m}} \phi_{1+\sigma_0^2}(x) \,\mathrm{d}x\leq\exp\biggl(-
\frac
{1}{8} M^2m \biggr) =3^{-2m}
\]
since
$ M^2 = 8 \log9$.

Combining these two upper bounds for the integral, we obtain
\[
\int\frac{(f_\alpha-f_\beta)^2}{f_0} \leq\sqrt{m} \varepsilon^2 \bigl(
2\uppi
C^\ast+ 64/3 \bigr)= \frac{c_1}{2n}
\]
as long as
%
%
\begin{equation}\label{boundsigmam}
\sqrt{m} \varepsilon^2 \leq\frac{1}{n}\frac{c_1}{ 2(2\uppi C^\ast
+64/3)}.
\end{equation}
%

As a consequence, the constructed mixing densities fulfil the two
requirements in Assouad's lemma under conditions (\ref{halfg}) and
(\ref
{boundsigmam}), with
\[
\varepsilon^2 \leq\min\biggl(\frac{1}{16^2}3^{-2m+1}m^{-3},
\frac
{1}{n\sqrt{m}} \frac{c_1}{2(2\uppi C^\ast+64/3)} \biggr).
\]
From Lemma \ref{assouadsim}, the lower bound is obtained as
$c\varepsilon
^2 m$,
which is at most
\[
\min\bigl(3^{-2m}m^{-2}, \sqrt{m}/n\bigr)
\]
up to a constant.
To find the largest $m \varepsilon^2$, by equating $3^{-2m}m^{-2} =
\sqrt{m}/n$, we obtain $m$ and $\varepsilon^2$ as $\log n$ and
$1/(n\sqrt{\log
n})$, respectively, up to a constant, and hence the lower bound is
obtained as $ \sqrt{\log n}/n$ up to a constant.
\end{pf*}

\subsection{Ideas in the proof of Theorem \texorpdfstring{\protect\ref{mainthm2}}{1.3}}
Here we let $W(f,g) := \Vert\sqrt{f}-\sqrt{g}\Vert_2^2 = \int(\sqrt
{f}-\sqrt{g})^2$, so (\ref{metricW}) is satisfied with $\zeta= 1$.
First, we relate the Hellinger distance and the $\chi^2$ distance.
That is, suppose we can show
$(1/2) \pi_0(u) \leq\pi_\varsigma(u) \leq(3/2) \pi_0(u)$ so that
$ (1/2)f_0(x) \leq f_\varsigma(x) \leq(3/2)f_0(x)$ for both
$\varsigma
=\alpha$ and $\varsigma= \beta$, by convolving with the standard
normal density.
Then, using the upper bound for $f_\alpha$ and $f_\beta$,
%
%
\begin{equation}\label{uppertheta}
\int( \sqrt{f_\alpha}-\sqrt{f_\beta} )^2
= \int\frac{(f_\alpha-f_\beta)^2}{(\sqrt{f_\alpha}+ \sqrt{f_\beta})^2}
\geq\frac{1}{6} \int\frac{(f_\alpha-f_\beta)^2}{f_0}. 
\end{equation}

Similarly, the lower bound for $f_\alpha$ would give an upper bound for
the testing condition
%
%
\begin{equation}\label{lowertheta}
\int\frac{(f_\alpha-f_\beta)^2}{f_\alpha} \leq2 \int\frac
{(f_\alpha
-f_\beta)^2}{f_0}.
\end{equation}
Thus it would be enough to work with the following quantity
\[
\int\frac{(f_\alpha-f_\beta)^2}{f_0} = \int\biggl( \frac{f_\alpha
}{\sqrt{f_0}} -\frac{f_\beta}{\sqrt{f_0}}
\biggr)^2 = \varepsilon^2 \int\biggl( \sum
_{k \in K} (\alpha_k-\beta_k)
\frac{\phi\star
v_k}{\sqrt{f_0}} \biggr)^2,
\]
where the second equality is given by (\ref{formG}).

At first glance, $\int(f_\alpha-f_\beta)^2/f_0$ does not look amenable
to Fourier techniques.
However, as Lemma \ref{absorbg0} below shows, $\phi\star v_k/\sqrt
{f_0}$ is expressed as convolution of a normal density (with a variance
larger than $1$) with a certain choice of the perturbation function
$v_k$ and base function $\pi_0 = \phi_{\sigma^2}$.

%
\begin{lemma}\label{absorbg0}
Consider the perturbation functions
\[
v_k(u)= \frac{C_k}{\sqrt{k!}} \phi(\rho u) H_k(\gamma u),\qquad
\rho^2 \geq\frac{1}{\sigma^2}+\frac{\gamma^2}{2},
\]
where $C_k$ is a constant depending on $k$ and $\gamma>0$.
Then
%
%
\begin{equation}\label{tildeconvol}
\frac{[\phi\star v_k](x)}{\sqrt{\phi\star\phi_{\sigma^2}(x)}}
= \phi_{\tilde\sigma^2} \star\tilde v_k,
\end{equation}
where
%
%
\begin{equation}\label{tildevk}
\tilde v_k(u) :=\frac{\tilde C_k}{\sqrt{k!}} \phi(\tilde\rho u) H_k
(\tilde\gamma u),
\end{equation}
with
%
%
\begin{equation}\label{tildedef}
\tilde\sigma^2 = 1+ \frac{1}{2\sigma^2+1}, \qquad\tilde C_k = C_k \frac{
(4\uppi)^{1/4} }{\tilde\sigma} ,\qquad \tilde\rho= \frac{\sqrt{\rho
^2+1-\tilde\sigma^2}}{\tilde\sigma^2},\qquad \tilde\gamma= \frac
{\gamma
}{\tilde\sigma^2}.
\end{equation}
\end{lemma}

By Lemma \ref{absorbg0}, the denominator effect can be incorporated
into the normal convolution.
Then we follow similar ideas used in the proof of Theorem \ref{mainthm}.


%
\begin{pf*}{Proof of Theorem \ref{mainthm2}}
Again, the choice of $v_k$'s is suggested by Fourier methods.
For convenience, we let $\pi_0 = \phi$, so $f_0 = \phi_{2}$ and
$\sqrt{f_0} = 2 \pi^{1/4} \phi_{4}$.
Assuming $\tilde v_k$ in (\ref{tildevk}) are in~$\LL_2$,
\begin{eqnarray*}
\fourier\biggl[\frac{f_\alpha}{\sqrt{f_0}} \biggr](t) &=&\fourier[
\sqrt{f_0}
](t) + \varepsilon\sum_{k \in K} \alpha_k
\fourier\biggl[ \frac{\phi\star v_k}{\sqrt{f_0}} \biggr](t)
\\
&=& \fourier[ \sqrt{f_0}](t) + \varepsilon\sum
_{k \in K} \alpha_k \fourier\phi_{{4}/{3}}(t)
\fourier\tilde v_k (t)
\end{eqnarray*}
by Lemma \ref{absorbg0}.

By the Plancherel formula,
\begin{eqnarray*}
\frac{1}{2\uppi} \biggl\llVert\frac{f_\alpha}{\sqrt{f_0}} - \frac
{f_\beta
}{\sqrt{f_0}} \biggr
\rrVert_2 
 = \varepsilon^2 \int
\biggl\llvert\sum_{k \in K} (\alpha_k-
\beta_k) \fourier\phi_{{4}/{3}}(t) \fourier[\tilde
v_k] (t) \biggr\rrvert^2 \,\mathrm{d}t,
\end{eqnarray*}
which lets us write the condition (\ref{cond1}) in Assouad's Lemma
\ref
{assouadsim} as
\[
\int\biggl\llvert\sum_{k \in K} (\alpha_k-
\beta_k) \fourier[\phi_{
{4}/{3}}](t) \fourier[\tilde
v_k] (t) \biggr\rrvert^2 \,\mathrm{d}t \geq\frac
{3c_0}{\uppi
} \sum
_{k \in K}(\alpha_k-\beta_k)^2\qquad
\forall\alpha, \beta\in\{ 0,1\}^K,
\]
with $\delta=\uppi c_0 \varepsilon^2/3$ by (\ref{uppertheta}).

Similar to the case for squared error loss, we might achieve even an
equality with $c_0=\uppi/3$ by choosing $\tilde v_k$'s to make the
functions $\psi_k(t) := \fourier[\phi_{{4}/{3}}](t) \fourier
[\tilde v_k] (t) $ orthonormal.
Ignoring other requirements, we also start from the same orthonormal
set (\ref{psikchoice}),
and then try to define $\tilde v_k$ as the inverse Fourier transform.

From the fact that
\[
\fourier[\phi_{{4}/{3}}](t) = \frac{1}{\sqrt{2\uppi}} \exp\biggl(-
\frac
{2}{3} t^2\biggr)
\]
and by definition of $\tilde v_k$ in (\ref{tildevk}),
the requirement is that
%
%
\begin{equation}\label{settingortho}
\psi_k(t) := \frac{\tilde C_k}{\sqrt{2\uppi k!}} \exp\biggl(-\frac{2 t^2}{3}\biggr)
\fourier\bigl[ \phi(\tilde\rho u) H_k(\tilde\gamma u)\bigr](t)
=
\mathrm{i}^{-k} \sqrt{2\phi(2t)} \frac{H_k(2t)}{\sqrt{k!}}.
\end{equation}
If we determine all the parameters to make (\ref{settingortho}) true,
we have the desired property for the loss separation condition (\ref
{cond1}), that is, we have
%
%
\begin{equation}\label{hellingerL2}
\int\frac{(f_\alpha-f_\beta)^2}{f_0} =2\uppi\varepsilon^2 \sum_{k \in K}
(\alpha_k-\beta_k)^2.
\end{equation}
%
We have to find $\tilde\rho$, $\tilde\gamma$ and $\tilde C_k$ so that
(\ref{settingortho}) is satisfied.
The solutions are derived below and given in (\ref{tildesol}).

After some calculations,
\[
\fourier\bigl[ \phi(\tilde\rho u) H_k(\tilde\gamma u)\bigr](t) =
\mathrm{i}^{-k} \frac{\sqrt{2}(2\uppi)^{3/4}}{\tilde C_k} \phi\biggl(
t\sqrt{\frac
{2}{3}}
\biggr) H_k(2t),
\]
which leads to
\[
\fourier^{-1} \biggl[ \phi\biggl(t\sqrt{\frac{2}{3}} \biggr)
H_k(2t) \biggr](u) = \mathrm{i}^{k}\frac{\tilde C_k}{\sqrt{2}(2\uppi
)^{3/4}} \phi(\tilde
\rho u) H_k(\tilde\gamma u).
\]
%
Substituting $a=\sqrt{2/3}$ and $b=2$ into Lemma \ref{hermiteFT},
we have the following solutions,
%
%
\begin{equation}\label{tildesol}
\tilde C_k = (2\uppi)^{3/4}\sqrt{3}\sqrt{5^k},\qquad \tilde\rho= \sqrt{\frac
{3}{2}},\qquad \tilde\gamma= \frac{3}{\sqrt{5}} .
\end{equation}

We need to ensure that the choice of $\sigma^2 =1$ satisfies the
inequality $\rho^2 \geq\frac{1}{\sigma^2} + \frac{\gamma^2}{2}$ needed
for the Lemma \ref{absorbg0}.
Comparing (\ref{tildedef}) and (\ref{tildesol}), we obtain
$\rho^2 = 3$ and $\gamma= \frac{4}{\sqrt{5}}$, which satisfy the condition.
Also, $C_k$ is obtained as $C_k = (2^{5/4} \sqrt{\uppi})\sqrt{5}^k$.

Therefore, this choice for the $\psi_k$'s leads to
%
%
\begin{equation}\label{choice2}
v_k(u) = 2^{5/4}\sqrt{\uppi}\sqrt{\frac{5^k}{k!}} \phi(\sqrt{3} u) H_k
\biggl( \frac{4}{\sqrt{5}} u \biggr)\qquad \mbox{for $k \in K$}.
\end{equation}
By restricting to odd values of $k$, we make the $v_k$'s real-valued
and odd, thereby ensuring that $\int v_k \,\mathrm{d}\lambda= 0$.

Using exactly the same idea as in the previous section, if
%
%
\begin{equation}\label{condhalfgH}
\varepsilon\leq\frac{1}{2\kappa mC_{2m-1}}\qquad \mbox{with $\kappa\simeq1.086$},
\end{equation}
then
%
%
\[
\tfrac{1}{2}\pi_0(u) \leq\pi_\alpha(u) \leq\tfrac{3}{2}\pi_0(u)\qquad
\mbox{for
all $u \in\rr, \alpha\in\{0,1\}^K$}.
\]

Now the second testing condition can be treated straightforwardly.
Indeed, once we choose orthonormal functions $\{\psi_k, k \in K\}$, we obtain
\[
\int\frac{(f_\alpha-f_\beta)^2}{f_\alpha} \leq\int2\frac
{(f_\alpha
-f_\beta)^2}{f_0}= 4\uppi\varepsilon^2\qquad
\mbox{for $\Vert\alpha-\beta\Vert_0=1$ by (\ref{lowertheta}) and (
\ref{hellingerL2})}.
\]
%
Thus it is enough to choose $\varepsilon^2 < 1/(4\uppi n)$.
With our choice $\varepsilon= 1/(4\sqrt{n})$, the testing condition is
satisfied.

From the lower bound $m\varepsilon^2$, we want to choose $m$ as large
as possible.
The condition in (\ref{condhalfgH}) restricts the size of $m$,
\[
2\kappa mC_{2m-1} <6m 5^m \leq4\sqrt{n}.
\]
Thus, we have the upper bound for $m$,
\[
m \lesssim\bigl(1/(2\log5)\bigr) \log n \simeq(0.31) \log n.
\]

Finally, we check these constructed $\pi_\alpha$'s are inside of the
parameter space $\PP_s(\rr)$.
From the fact that $\pi_\alpha(u) \leq(3/2)\pi_0(u)$ for all $u \in
\rr
$ and $\alpha\in\{0,1\}^K$,
it is clear that $\pi_0$ is in the space $\PP_s(\rr)$ from the tail
property of normal density.

Consequently, the lower bound is obtained as $\log n/n$ up to a constant.
\end{pf*}
%
%

\section{Proof of the upper bound}\label{upperbound}
For the reader's convenience, we summarize the arguments for Theorem
\ref{ibragimovthm}, following pages~365--369 by \cite{ibragimov2001}.
Before turning to that result, we first show that if $f = \phi\star
\Pi\in\FF$, then $f$ can be extended to an entire function $f^*$.
To see this, we let $f^*(x+\mathrm{i}y) =\frac{1}{\sqrt{2\uppi}} \int
\exp
(-(x+\mathrm{i}y-u)^2)/2 )\,\mathrm{d}\Pi(u)$. Defining $a(x,y,u)
\equiv y(u-x), b(x,y,u) \equiv-\frac{1}{2} \{(x-u)^2-y^2 \}$, we write
\begin{eqnarray*}
\sqrt{2\uppi}f^*(x+\mathrm{i}y)&=&\int\bigl[ \bigl\{\cos a (x,y,u ) +
\mathrm{i} \sin a ( x,y,u )
\bigr\} \mathrm{e}^{b(x,y,u)} \bigr]\,\mathrm{d}\Pi(u)
\\
&=& \int\bigl\{ \cos a (x,y,u )\mathrm{e}^{b(x,y,u)} \bigr\} \,\mathrm
{d}\Pi(u) + \mathrm{i} \int\bigl
\{\sin a (x,y,u )\mathrm{e}^{b(x,y,u)} \bigr\} \,\mathrm{d}\Pi(u)
\\
&:=& v(x,y) + \mathrm{i}w(x,y).
\end{eqnarray*}
By differentiating under the integral (see Theorem 16.8 in \cite{billingsley}),
\begin{eqnarray*}
\frac{\partial v}{\partial x} &=&\int\bigl[ \bigl\{ y \sin a(x,y,u)+
(u-x)\cos a(x,y,u)
\bigr\} \mathrm{e}^{b(x,y,u)} \bigr] \,\mathrm{d}\Pi(u) = \frac
{\partial
w}{\partial y},
\\
\frac{\partial v}{\partial y} &= &\int\bigl[ \bigl\{ y\cos
a(x,y,u)-(u-x)\sin a(x,y,u)
\bigr\} \mathrm{e}^{b(x,y,u)} \bigr] \,\mathrm{d}\Pi(u)= -\frac
{\partial w}{\partial x}.
\end{eqnarray*}
Also note that $\partial v/\partial x,\partial v/\partial y, \partial
w/\partial x$, and $\partial w/\partial y$ are continuous. Then by
Cauchy--Riemann theorem (see Theorem 1.5.8 in \cite{marsden1999}),
$f^*$ is analytic.

Now it suffices to show that $f^*$ satisfies the growth condition.
Indeed,
%
%
%
\begin{eqnarray}
\label{analyticcond} \sup_{x} \bigl\llvert f^*(x+\mathrm{i}y)\bigr
\rrvert&= &\sup_x \biggl\llvert\frac{1}{\sqrt{2\uppi
}}\int\exp
\biggl( -\frac{(x+\mathrm{i}y-u)^2}{2} \biggr) \,\mathrm{d}\Pi(u) \biggr
\rrvert
\nonumber
\\
&\leq&\frac{1}{\sqrt{2\uppi}} \sup_x\int\biggl\llvert\exp
\biggl( -\frac
{(x+\mathrm{i}y-u)^2}{2} \biggr) \biggr\rrvert\,\mathrm{d}\Pi(u)
\nonumber
\\[-8pt]
\\[-8pt]
\nonumber
&\leq&\frac{1}{\sqrt{2\uppi}} \exp\biggl(
\frac{y^2}{2} \biggr) \sup_x \int\,\mathrm{d}\Pi(u) \\
&=&
\frac{1}{\sqrt{2\uppi}} \exp\biggl( \frac
{y^2}{2} \biggr).\nonumber
\end{eqnarray}
Thus, $\FF\subseteq\FF^*$, which ensures that Ibragimov's estimation
also gives the upper bound to match Theorem \ref{mainthm}.

%
\begin{pf*}{Proof of Theorem \ref{ibragimovthm}}
Ibragimov used a \textit{sinc} kernel estimator,
\[
\hat f_n(x)= \frac{1}{nh} \sum_{j=1}^n
\KK\biggl( \frac
{X_j-x}{h} \biggr),\qquad \KK(u) = \frac{\sin(u)}{\uppi u},
\]
with $h = 1/\sqrt{\log n}$.
It is important for his method that the Fourier transform of $\KK$ is
$\breve\KK(t) =\frac{1}{\sqrt{2\uppi}} \one\{|t|\leq1\}$ and also
$\breve\KK^2(t) =\frac{1}{\sqrt{2\uppi}} \frac{1}{\uppi}
(1-\frac
{|t|}{2} )_{+}$ where $x_{+} = \max(x,0)$.

The expected mean integrated squared error (MISE) has the usual squared
bias and variance decomposition.
As usual, the variance term is bounded by $(nh)^{-1} \int\KK^2(u)\,
\mathrm{d}u$.
For the bias term, note that $\ee_{n,f} \hat f$ has the Fourier
transform $\sqrt{2\uppi} \breve f(t) \breve\KK(ht)$, so that
\begin{eqnarray*}
\operatorname{bias}^2 &:=& \int(\ee_{n,f}\hat
f_n - f )^2 = \int\bigl\llvert\fourier[
\ee_{n,f} \hat f_n](t) - \fourier[f](t) \bigr\rrvert
^2 \,\mathrm{d}t \qquad\mbox{by Plancherel}
\\
&=& \int\bigl\llvert\breve f(t) \bigr\rrvert^2 \bigl\llvert\sqrt{2
\uppi} \breve\KK(ht) - 1 \bigr\rrvert^2 \,\mathrm{d}t
\\
&=& \int_{|t| \geq1/h} \bigl\llvert\breve f(t) \bigr\rrvert
^2 \,\mathrm{d}t \qquad\mbox{by the form of } \breve\KK
\\
&\leq& 2 \mathrm{e}^{-y/h} \int\mathrm{e}^{-yt} \bigl\llvert\breve
f(t)\bigr
\rrvert^2 \,\mathrm{d}t\qquad \mbox{for } y>0
\\
&=&2 \mathrm{e}^{-y/h} \lim_{M \rightarrow\infty} \int\mathrm{e}^{-yt}
\bigl
\llvert\breve f(t) \bigr\rrvert^2 \biggl(1-\frac{|t|}{M}
\biggr)_{+} \,\mathrm{d}t.
\end{eqnarray*}
Write the last integral as
\begin{eqnarray*}
\int\mathrm{e}^{-yt} \bigl\llvert\breve f(t) \bigr\rrvert^2
\biggl( 1-\frac{|t|}{M} \biggr)_{+}\,\mathrm{d}t = \int\breve f(t)
\mathrm{e}^{-yt} \overline{\breve\vartheta(t)}\,\mathrm{d}t = \int
\breve f(t)
\mathrm{e}^{-yt}\breve\vartheta(-t) \,\mathrm{d}t, 
\end{eqnarray*}
where $\breve\vartheta(t) =\breve f(t) (1-\frac{|t|}{M}
)_{+}$ is the Fourier transform of the nonnegative function $\vartheta
(x) = \frac{M\pi^2}{\sqrt{2\uppi}} \int f(u) \KK^2 (\frac{M}{2}(x-u)
) \,\mathrm{d}u$.
Using $\vartheta(x+\mathrm{i}y) = \frac{1}{\sqrt{2\uppi}} \int
\mathrm{e}^{\mathrm{i}tx}\mathrm{e}^{-yt}\breve\vartheta(t)\,\mathrm
{d}t$, we have $\int f(x) \vartheta
(x+\mathrm{i}y)\,\mathrm{d}x = \int\breve f(t) \mathrm{e}^{yt}\breve
\vartheta(-t)\,\mathrm{d}t$ by
Parseval's theorem.
By changing the contour of the integration, $\int f(x+\mathrm{i}y)
\vartheta(x)
\,\mathrm{d}x = \int\breve f(t) \mathrm{e}^{-yt}\breve\vartheta(-t)\,
\mathrm{d}t$.
Combining these ideas,
\[
\int\breve f(t) \mathrm{e}^{-yt}\vartheta(-t)\,\mathrm{d}t = \int
f(x+\mathrm{i}y) \vartheta(x)\,\mathrm{d}x
\leq\int\sup_x \bigl\llvert f(x+\mathrm{i}y) \bigr\rrvert
\vartheta(x) \,\mathrm{d}x \leq\frac
{\exp
(y^2/2)}{\sqrt{2\uppi}},
\]
where the last inequality follows by (\ref{analyticcond}) together with
$\int\vartheta(x)\,\mathrm{d}x = \sqrt{2\uppi}\breve\vartheta(0)=1$.
By taking $y = 1/h$, we obtain the upper bound as $\sqrt{\log n}/n :=
\ell_n$ up to a constant.
\end{pf*}
\section{Discussion}
It has been claimed that the Fano's method is more general in a sense
(see \cite{yu97lecam}, page 428).
Indeed, using Varshamov--Gilbert's lemma (e.g., Lemma 2.9 in \cite
{tsybakov2009}), it is not very difficult to prove the same rate result
for $\FF$ with a similar type of sub-parameter space using Fano's method.

However, Assouad's method seems more convenient in some cases.
For instance, before knowing how to construct the subspace, it would be
extremely difficult to determine the right family of densities when
there are only indirect regularity conditions as in this example.
Assouad's hyper-rectangle method indicates that the problem can be
solved if we can show the orthogonality relations between the
constructed densities.
These added regularity conditions can cause different difficulties, but
we at least have some clues to handle these problems.

On the other hand, if we know metric entropy (good packing and covering
number bounds) results beforehand, the optimal minimax rates can be
obtained almost automatically with the predictive Bayes density
estimator using the main theorems in \cite{yangbarron1999}.
It will be interesting to see if we can calculate a sharper metric
entropy for $\FF$ or $\FF_s$ than the one that appeared in \cite
{ghosalvaart2001}.

\begin{appendix}\label{app}
\section*{Appendix}
\begin{pf*}{Proof of Lemma \ref{assouadsim}}
Most of the proof is based on ideas borrowed from \cite{lecam73,tsybakov2009},
and some unpublished notes by David Pollard. Denote $A=\{0,1\}^K$ and
for convenience denote $\ee_\alpha$ for $\ee_{f_\alpha}$ and
${\mathbb P}_\alpha
$ for ${\mathbb P}_{f_\alpha}$ where ${\mathbb P}_{f_\alpha} =
P_{f_\alpha}^n$. For any
density estimator $\hat f_n$ based on the observation $X_1,\ldots,X_n$,
define an estimator
\[
\hat\alpha= \mathop{\arg\min}_{\alpha\in A} W(\hat f_n,
f_\alpha).
\]
%
By restricting the parameter space and by the definition of $\hat
\alpha$,
\begin{eqnarray*}
\sup_{f \in\FF} \ee_f W(\hat f_n, f) &\geq&
\max_{\alpha\in A} \ee_\alpha W (\hat f_n,f_\alpha)
\\
&\geq&\frac{1}{2}\max_{\alpha\in A} \ee_\alpha\bigl(W(
\hat f_n,f_\alpha)+ W(\hat f_n,f_{\hat\alpha})
\bigr)
\\
&\geq&\frac{\zeta}{2} \max_{\alpha\in A} \ee_\alpha W
(f_\alpha,f_{\hat
\alpha})
\end{eqnarray*}
using the pseudo-distance property (\ref{metricW}).
Now, using the condition (\ref{cond1}) in the lemma followed by the
simple fact that the supremum is bounded by the average, the last
equation can be lower bounded by
\[
\frac{c_0 \varepsilon^2 \zeta}{2} \max_{\alpha\in A} \sum
_{k=1}^m \ee_\alpha\one\{
\alpha_k \neq\hat\alpha_k \} \geq\frac{c_0
\varepsilon
^2 \zeta}{2}
\frac{1}{2^m} \sum_{\alpha\in A} \sum
_{k=1}^m \ee_\alpha\one\{
\alpha_k \neq\hat\alpha_k\}.
\]

Define
\[
\bar{\mathbb P}_{0,k}=\frac{1}{2^{m-1}} \sum
_{\alpha\in A_{0,k}} {\mathbb P}_\alpha,\qquad \bar{\mathbb
P}_{1,k}=\frac{1}{2^{m-1}} \sum_{\alpha\in
A_{1,k}} {
\mathbb P} _\alpha,\qquad  k=1,\ldots,m,
\]
where $A_{i,k}=\{\alpha\in A\dvt\alpha_k=i\}$ for $i=0,1$.

Since $\alpha_k, \hat\alpha_k\in\{0,1\}$, we have
\begin{eqnarray*}
\frac{1}{2^m}\sum_{\alpha\in A} \sum
_{k=1}^m \ee_\alpha\one\{ \alpha
_k\neq\hat\alpha_k\} 
&=&\frac{1}{2^m} \sum_{k=1}^m \biggl(
\sum_{\alpha\in A_{0,k}} {\mathbb P} _{\alpha} \one\{ \hat
\alpha_k \neq0\} + \sum_{\alpha\in A_{1,k}} {\mathbb P}
_{\alpha} \one\{\hat\alpha_k \neq1\} \biggr)
\\
&= &\frac{1}{2} \sum_{k=1}^m \bigl(
\bar{\mathbb P}_{0,k} \one\{ \hat\alpha_k \neq0\} + \bar{
\mathbb P}_{1,k} \one\{ \hat\alpha_k \neq1\} \bigr),
\end{eqnarray*}
which gives us the following lower bound
\[
\sup_{f \in\FF} \ee_{f} W (\hat f_n , f)
\geq\frac{c_0 \varepsilon^2 \zeta}{4} \sum_{k=1}^m\Vert
\bar{\mathbb P}_{0,k} \wedge\bar{\mathbb P}_{1,k} \Vert
_1
\]
by $Ph+Q(1-h) \geq\Vert P \wedge Q\Vert_1$ for $h \geq0$ with
$h=\mathbh
{1}\{
\hat\alpha\neq0\}$.

For $k=m$, each $\alpha$ in $A_{0,m}$ is of the form $(\gamma,0)$ with
$\gamma\in D:= \{0,1\}^{m-1}$. Similarly, each $\alpha$ in $A_{1,m}$
is of the form $(\gamma,1)$ with $\gamma\in D$. Now
\[
\Vert\bar{\mathbb P}_{0,m} \wedge\bar{\mathbb P}_{1,m} \Vert
_1=\int\biggl(\frac
{1}{2^{m-1}} \sum_{\gamma\in D}
p_{\gamma,0} \biggr) \wedge\biggl( \frac{1}{2^{m-1}} \sum
_{\gamma\in D} p_{\gamma,1} \biggr) \geq\int\frac{1}{2^{m-1}} \sum
_{\gamma\in D} (p_{\gamma,0} \wedge p_{\gamma,1}).
\]
Note that $(\gamma,0)$ and $(\gamma,1)$ have only one different
coordinate. By similar calculations for other $k'$s, we obtain
\[
\sup_{f \in\FF} \ee_f W(\hat f_n, f) \geq
\frac{c_0 \varepsilon^2
\zeta
}{4}m \min_{d(\alpha, \beta)=1} \Vert{\mathbb P}_\alpha
\wedge{\mathbb P}_\beta\Vert_1.
\]

In general, it is difficult to calculate the testing affinity exactly.
Fortunately, a convenient lower bound in terms of distances between
marginals is available when ${\mathbb P}_\alpha$ and ${\mathbb
P}_\beta$ are both
product measures. For instance, when ${\mathbb P}_\alpha= P_{\alpha
}^n$ for
i.i.d. case, we can bound this using the chi-squared distance $\chi^2$
by the following relation.
\[
\bigl(1-\Vert{\mathbb P}_\alpha\wedge{\mathbb P}_\beta\Vert
_1\bigr)^2 \leq n \chi^2(P_\alpha,
P_\beta) :=n \int\frac{(\theta_\alpha-\theta_\beta)^2}{\theta_\alpha}.
\]

Thus, the condition (\ref{cond2}) in the lemma yields a lower bound for
the maximum risk
\[
\sup_{f \in\FF} \ee_f W(\hat f_n, f) \geq
\frac{c_0\varepsilon^2
\zeta
}{4}m(1-\sqrt{c_1}).
\]
\upqed\end{pf*}

See \cite{tsybakov2009}, Lemma 2.7 on page 90, or \cite{lecam73},
Lemma 1 on page 40, for the derivation of facts about relations between
distances.

\begin{pf*}{Proof of Lemma \ref{hermiteFT}}
For $b>a >0$, we have
\[
\phi(at) \exp\biggl(btx- \frac{1}{2}x^2 \biggr) = \phi(at)
\sum_{k=0}^\infty\frac{H_k(bt)}{k!}x^k.
\]
Thus,
\begin{eqnarray*}
\fourier^{-1} \biggl[ \phi(at) \exp\biggl( bt x -
\frac{1}{2}x^2 \biggr) \biggr] (u) &=& \int_{-\infty}^\infty
\frac{\exp( \mathrm{i} tu)}{2\uppi}\exp\biggl(-\frac{a^2
t^2}{2} + bxt - \frac{1}{2}x^2
\biggr) \,\mathrm{d}t
\\
&=& \frac{1}{a\sqrt{2\uppi}}\exp\biggl(
\frac{(bx + \mathrm{i} u)^2}{2a^2} - \frac
{1}{2} x^2 \biggr)
\\
&=&\frac{1}{a} \phi\biggl(\frac{u}{a} \biggr) \exp\biggl(
\frac{ bx u
\mathrm{i}}{a^2} - \frac{1}{2} (\mathrm{i} x c_{a,b} )^2
\biggr)
\\
&=&\frac{1}{a}\phi\biggl(\frac{ u}{a} \biggr)\sum
_{k=0}^\infty\frac
{ H_k
( b/(a^2 c_{a,b}) u )}{k!} (\mathrm{i} c_{a,b}
)^{k} x^k.
\end{eqnarray*}
The inverse Fourier transform of the right side is
\[
\sum_{k=0}^\infty\fourier^{-1}
\biggl[ \phi(at) \frac{H_k(b t)}{k!} \biggr] (u) x^k.
\]
By matching the coefficient for the $k${th} power of $x$,
\[
\fourier^{-1} \bigl[ \phi(at) H_k(b t) \bigr](u) = (\mathrm{i}
c_{a,b} )^{k}\frac{1}{a}\phi\biggl(\frac{u}{a}
\biggr) H_k \biggl( \frac{b}{a^2
c_{a,b}} u \biggr) ,
\]
%
which proves the claim.
\end{pf*}


\begin{pf*}{Proof of Lemma \ref{absorbg0}}
First, note that $\phi\star\phi_{\sigma^2} = \phi_{1+\sigma^2}$.
We define $[\phi\star v_k(u)](x) = \int\phi(x-u)v_k(u)\,\mathrm{d}u$ and
similarly $[\phi\star\phi(\rho u) H_k(\gamma u)](x) = \int\phi(x-u)
\phi(\rho u) H_k(r u) \,\mathrm{d}u$.
By definition of $v_k$, we have
\begin{eqnarray*}
\frac{[\phi\star v_k(u)](x)}{\sqrt{\phi_{1+\sigma^2}(x)}} &=& \frac
{C_k}{\sqrt{k!}}\frac{[\phi\star\phi(\rho u) H_k(\gamma u)]
(x)}{\sqrt{\phi_{1+\sigma^2}(x)}}
\\
&=& \frac{C_k}{\sqrt{k!}} \int\frac{({1}/{\sqrt{2\uppi}})\exp
(-
{1}/{2}(x-u)^2) ({1}/{\sqrt{2\uppi}}) \exp(-({1}/{2})\rho^2 u^2)
H_k(\gamma u)}{(2\uppi(1+\sigma^2))^{-1/4} \exp(-({1}/{4})
({x^2}/{(1+\sigma^2)}))} \,\mathrm{d}u. 
\end{eqnarray*}
%
Now, by completing the square,
\begin{eqnarray*}
&&\exp \biggl(-\frac{1}{2}(x-u)^2 \biggr)\exp\biggl(-
\frac{1}{2}\rho^2 u^2 \biggr) \exp\biggl(
\frac{1}{4} \frac{x^2}{1+\sigma^2} \biggr)
\\
&&\qquad=\exp\biggl( \biggl(-\frac{1}{2}+\frac{1}{4(1+\sigma^2)} \biggr)
x^2 + xu - \biggl(\frac{1}{2}+ \frac{1}{2}
\rho^2 \biggr)u^2 \biggr)
\\
&&\qquad= \exp\biggl( -\frac{1}{2\tilde\sigma^2}x^2+ xu - \biggl(
\frac{1}{2}+ \frac{1}{2}\rho^2 \biggr)u^2
\biggr)\qquad \mbox{by definition of $\tilde\sigma^2$ in (\ref{tildedef})}
\\
&&\qquad=\exp\biggl(-\frac{1}{2 \tilde\sigma^2} (x-\tilde u)^2 \biggr) \exp
\biggl( -\frac{1}{2}\bigl(1+\rho^2-\tilde\sigma^2
\bigr)\frac{\tilde
u^2}{\tilde
\sigma^4} \biggr)\qquad \mbox{by $\tilde u := \tilde\sigma^2
u$}
\\
&&\qquad= (2\uppi\tilde\sigma) \phi_{\tilde\sigma^2}(x-\tilde u) \phi\biggl(
\frac{ \sqrt{1+\rho^2-\tilde\sigma^2}}{\tilde\sigma^2} \tilde u \biggr),
\end{eqnarray*}
where the positive value for $(1+\rho^2-\tilde\sigma^2)$ is guaranteed
by the condition $\rho^2 \geq1/\sigma^2+\gamma^2/2 > 1/(1+2\sigma^2)
:= 1-\tilde\sigma^2$.
By change of variables,
\[
\frac{[\phi\star v_k(u)](x)}{\sqrt{\phi_{1+\sigma^2}(x)}} 
= \biggl(\frac{C_k}{\sqrt{k!}}
\frac{ [2\uppi(1+\sigma^2)]^{1/4}
}{\tilde
\sigma} \biggr) \phi_{\tilde\sigma^2} \star\phi\biggl(
\frac{
\sqrt{1+\rho^2-\tilde\sigma^2}}{\tilde\sigma^2} \tilde u \biggr) H_k
\biggl( \frac{\gamma}{\tilde\sigma^2}
\tilde u \biggr).
\]
%
Using the definitions of each transformed variables (\ref{tildedef}),
the proof is complete.
\end{pf*}
\end{appendix}
\section*{Acknowledgements}
This work is part of the author's Ph.D. dissertation, written at Yale University.
The author is grateful to David Pollard, Harrison Zhou, and Richard J.
Samworth for their comments and advice. The research was supported in part by NSF Career Award DMS-06-45676 and NSF FRG Grant
DMS-08-54975.

%
%

%



\printhistory


\begin{thebibliography}{14}

\bibitem{billingsley}
%
\begin{bbook}[mr]
\bauthor{\bsnm{Billingsley},~\bfnm{Patrick}\binits{P.}}
(\byear{1995}).
\btitle{Probability and Measure},
\bedition{3rd} ed.
\bseries{Wiley Series in Probability and Mathematical Statistics}.
\blocation{New York}: \bpublisher{Wiley}.
\bid{mr={1324786}}
\bptok{imsref}%
\end{bbook}
%
\endbibitem

\bibitem{ghosalvaart2001}
%
\begin{barticle}[mr]
\bauthor{\bsnm{Ghosal},~\bfnm{Subhashis}\binits{S.}} \AND
\bauthor{\bparticle{van~der} \bsnm{Vaart},~\bfnm{Aad~W.}\binits{A.W.}}
(\byear{2001}).
\btitle{Entropies and rates of convergence for maximum likelihood and {B}ayes
estimation for mixtures of normal densities}.
\bjournal{Ann. Statist.}
\bvolume{29}
\bpages{1233--1263}.
\bid{doi={10.1214/aos/1013203453}, issn={0090-5364}, mr={1873329}}
\bptok{imsref}%
\end{barticle}
%
\endbibitem

\bibitem{table}
%
\begin{bbook}[mr]
\bauthor{\bsnm{Gradshteyn},~\bfnm{I.~S.}\binits{I.S.}} \AND
\bauthor{\bsnm{Ryzhik},~\bfnm{I.~M.}\binits{I.M.}}
(\byear{2007}).
\btitle{Table of Integrals, Series, and Products},
\bedition{7th} ed.
\blocation{Amsterdam}: \bpublisher{Elsevier/Academic Press}.
\bid{mr={2360010}}
\bptok{imsref}%
\end{bbook}
%
\endbibitem

\bibitem{ibragimov2001}
%
\begin{bincollection}[mr]
\bauthor{\bsnm{Ibragimov},~\bfnm{I.}\binits{I.}}
(\byear{2001}).
\btitle{Estimation of analytic functions}.
In \bbooktitle{State of the Art in Probability and Statistics ({L}eiden,
1999)}
(\beditor{\binits{C.}\bfnm{C.} \bsnm{Klaasen}},
\beditor{\binits{M.}\bfnm{M.} \bsnm{de Gunst}}
\AND
\beditor{\binits{A.W.}\bfnm{A.~W.} \bparticle{van~der} \bsnm{Vaart}},
eds.).
\bseries{Institute of Mathematical Statistics Lecture
Notes---Monograph Series}
\bvolume{36}
\bpages{359--383}.
\blocation{Beachwood, OH}: \bpublisher{IMS}.
\bid{doi={10.1214/lnms/1215090078}, mr={1836570}}
\bptok{imsref}%
\end{bincollection}
%
\endbibitem

\bibitem{fourierjackson}
%
\begin{bbook}[mr]
\bauthor{\bsnm{Jackson},~\bfnm{Dunham}\binits{D.}}
(\byear{2004}).
\btitle{Fourier Series and Orthogonal Polynomials}.
\blocation{Mineola, NY}: \bpublisher{Dover}.
\bid{mr={2098657}}
\bptok{imsref}%
\end{bbook}
%
\endbibitem

\bibitem{kawata}
%
\begin{bbook}[mr]
\bauthor{\bsnm{Kawata},~\bfnm{Tatsuo}\binits{T.}}
(\byear{1972}).
\btitle{Fourier Analysis in Probability Theory}.
\blocation{New York}: \bpublisher{Academic Press}.
\bid{mr={0464353}}
\bptok{imsref}%
\end{bbook}
%
\endbibitem

\bibitem{lecam73}
%
\begin{barticle}[mr]
\bauthor{\bsnm{Le Cam},~\bfnm{L.}\binits{L.}}
(\byear{1973}).
\btitle{Convergence of estimates under dimensionality restrictions}.
\bjournal{Ann. Statist.}
\bvolume{1}
\bpages{38--53}.
\bid{issn={0090-5364}, mr={0334381}}
\bptok{imsref}%
\end{barticle}
%
\endbibitem

\bibitem{marsden1999}
%
\begin{bbook}[mr]
\bauthor{\bsnm{Marsden},~\bfnm{Jerrold~E.}\binits{J.E.}} \AND
\bauthor{\bsnm{Hoffman},~\bfnm{Michael~J.}\binits{M.J.}}
(\byear{1987}).
\btitle{Basic Complex Analysis},
\bedition{2nd} ed.
\blocation{New York}: \bpublisher{Freeman}.
\bid{mr={0913736}}
\bptnote{check year}%
\bptok{imsref}%
\end{bbook}
%
\endbibitem

\bibitem{rudinrealcomplex}
%
\begin{bbook}[mr]
\bauthor{\bsnm{Rudin},~\bfnm{Walter}\binits{W.}}
(\byear{1987}).
\btitle{Real and Complex Analysis},
\bedition{3rd} ed.
\blocation{New York}: \bpublisher{McGraw-Hill}.
\bid{mr={0924157}}
\bptok{imsref}%
\end{bbook}
%
\endbibitem

\bibitem{tsybakov2009}
%
\begin{bbook}[mr]
\bauthor{\bsnm{Tsybakov},~\bfnm{Alexandre~B.}\binits{A.B.}}
(\byear{2009}).
\btitle{Introduction to Nonparametric Estimation}.
\bseries{Springer Series in Statistics}.
\blocation{New York}: \bpublisher{Springer}.
\bid{doi={10.1007/b13794}, mr={2724359}}
\bptok{imsref}%
\end{bbook}
%
\endbibitem

\bibitem{vanassouad}
%
\begin{bbook}[mr]
\bauthor{\bparticle{Van~der} \bsnm{Vaart},~\bfnm{A.~W.}\binits{A.W.}}
(\byear{1998}).
\btitle{Asymptotic Statistics}.
\bseries{Cambridge Series in Statistical and Probabilistic Mathematics}
\bvolume{3}.
\blocation{Cambridge}: \bpublisher{Cambridge Univ. Press}.
\bid{mr={1652247}}
\bptok{imsref}%
\end{bbook}
%
\endbibitem

\bibitem{yangbarron1999}
%
\begin{barticle}[mr]
\bauthor{\bsnm{Yang},~\bfnm{Yuhong}\binits{Y.}} \AND
\bauthor{\bsnm{Barron},~\bfnm{Andrew}\binits{A.}}
(\byear{1999}).
\btitle{Information-theoretic determination of minimax rates of convergence}.
\bjournal{Ann. Statist.}
\bvolume{27}
\bpages{1564--1599}.
\bid{doi={10.1214/aos/1017939142}, issn={0090-5364}, mr={1742500}}
\bptok{imsref}%
\end{barticle}
%
\endbibitem

\bibitem{yu97lecam}
%
\begin{bincollection}[mr]
\bauthor{\bsnm{Yu},~\bfnm{Bin}\binits{B.}}
(\byear{1997}).
\btitle{Assouad, {F}ano, and {L}e {C}am}.
In \bbooktitle{Festschrift for {L}ucien {L}e {C}am}
(\beditor{\binits{D.}\bfnm{D.} \bsnm{Pollard}},
\beditor{\binits{E.}\bfnm{E.} \bsnm{Torgersen}}
\AND
\beditor{\binits{G.L.}\bfnm{G.~L.} \bsnm{Yang}}, eds.)
\bpages{423--435}.
\blocation{New York}: \bpublisher{Springer}.
\bid{mr={1462963}}
\bptok{imsref}%
\end{bincollection}
%
\endbibitem

\end{thebibliography}
\end{document}